\documentclass[11pt]{amsart}

\usepackage{amsmath}
\usepackage{amssymb}
\usepackage{graphicx}

\hoffset = - 0.5in

\newtheorem{theorem}{Theorem}[section]

\newtheorem{lemma}[theorem]{Lemma}
\newtheorem{proposition}[theorem]{Proposition}
\theoremstyle{definition}
\newtheorem{definition}[theorem]{Definition}

\numberwithin{equation}{section}

\renewcommand{\geq}{\geqslant}
\renewcommand{\leq}{\leqslant}

\title{Property($K^*$) Implies the Weak Fixed Point Property}

\author[Tim Dalby]{Tim Dalby}

\date{\today}



\email{tim\_dalby@bigpond.com}

\begin{document}

\parindent = 0pt
\parskip = 8pt

\begin{abstract}

It is shown that if the dual of a separable Banach space has Property($K^*$) then the original space has the weak fixed point property.  This is an improvement of previously results.

\end{abstract}

\maketitle

\section{Introduction}

A Banach space, $X$, has the weak fixed point property, w-FPP, if every nonexpansive mapping, $T$, on every weak compact convex nonempty subset, $C$, has a fixed point.  If the subsets are closed and bounded instead of weak compact then the property is called the fixed point property, FPP.  The past 40 or so years has seen a number of Banach space properties shown to imply the w-FPP.  Some such properties are weak normal structure, Opial's condition, Property($K$)  and Property($M$).

Definitions of relevant terms and moduli are given in the next section.

In [4] Dalby showed that if $X$ is a separable Banach space where the dual, $X^*$, has Property($K^*$) and the nonstrict *Opial condition then $R(X) < 2$.  This last condition implies $X$ has the w-FPP.  For the latter result see [9].

In [1] Benavides showed that if $X$ is a reflexive Banach space and if there exist $c_0 \in (0, 1)$ where $r_{X^*}(c_0) > 0$ then $R(a, X) \leq \max \left \{ 1 + ac_0, a + \frac{\displaystyle 1}{\displaystyle 1 + r_{X^*}(c_0)} \right \}$ where $a \geq 0.$  This in turn implies $R(a, X) < 1 + a \mbox{ if } r_{X^*}(1) > 0.$  Here $r(c)$ is Opial's modulus. Earlier in that paper it was shown that $R(a, X) < 1 + a \mbox{ for some } a > 0$ ensures that $X$ has the w-FPP.  Because of reflexivity, in this case $X$ has the FPP.

Recall from [7] that $R(a, X) < 1 + a$ for some $a > 0$ is equivalent to $R(1, X) < 2$ and that $R(X) < 2$ implies $R(1, X) < 2.$

First, it is proved below that the condition that there exist $c_0 \in (0, 1)$ where $r_{X^*}(c_0) > 0$ is equivalent to $r_{X^*}(1) > 0.$

So the result in [1] can be rewritten as: If $X$ is a reflexive Banach space where $r_{X^*}(1) > 0$ then $X$ has the FPP.   But Dalby proved in [3] that $r_{X}(1) > 0$ is equivalent to $X$ having Property($K$).  This proof can readily be transferred to $X^*$.  That is $r_{X^*}(1) > 0$ is equivalent to $X^*$ having Property($K^*$).  Hence the result in [1] can be again rewritten as: If $X$ is a reflexive Banach space where $X^*$ has Property($K^*$) then $X$ has the FPP.

So the results of Benavides and Dalby are similar.

Second, this paper shows that if $X$ is a separable Banach space where the dual, $X^*$, has Property($K^*$) then $X$ has the w-FPP.   Thus the nonstrict *Opial condition has been removed and reflexive changed to separable.  The proof is also direct, bypassing reference to $R(X)$ or $R(a, X) .$  The idea is to make the mathematics behind the conclusion a little clearer.

\section{Definitions}

\begin{definition}

Sims, [15]

A Banach space $X$ has property($K$) if there exists $K \in [0, 1)$ such that whenever $x_n \rightharpoonup 0, \lim_{ n \rightarrow \infty} \| x_n \| = 1 \mbox{ and }  \liminf_{n \rightarrow \infty} \| x_n - x \|  \leq 1 \mbox{ then } \| x \| \leq K.$

If the sequence is in $B_{X^*}$ and is weak* convergent to zero then the property is called Property($K^*$).

\end{definition}

\begin {definition}

Opial [14] 

A Banach space has Opial's condition if
\[ x_n \rightharpoonup 0 \ \mbox {and } x \not = 0 \mbox { implies } \limsup_n \| x_n \| < \limsup_n \| x_n - x \|. \]
The condition remains the same if both the $\limsup$s are replaced by $\liminf$s.

If the $<$ is replaced by $\leq$ then the condition is called nonstrict Opial.

In $X^*$ the conditions are the same except that the sequence is weak* null.

\end {definition}

Later a modulus was introduced to gauge the strength of Opial's condition and a stronger version of the condition was defined.

\begin{definition}

Lin, Tan and Xu, [13]

Opial's modulus is
\[ r_X(c) = \inf \{ \liminf_{n\rightarrow \infty} \| x_n - x \| - 1: c \geq 0, \| x \| \geq c,  x_n \rightharpoonup 0 \mbox{ and }\liminf_{n\rightarrow \infty} \| x_n \| \geq 1 \}. \]
$X$ is said to have uniform Opial's condition if $r_X(c) > 0$ for all $c > 0.$  See [13] for more details.

\end{definition}

\section{Results}

\begin{lemma}
Let $X$ be a Banach space then the condition that there exists a $c_0 \in (0, 1)$ where $r_X(c_0) > 0$ is equivalent to $r_X(1) > 0.$
\end{lemma}

\begin{proof}

Note that $r_X(c)$ is a nondecreasing function of $c.$  So if there exists a $c_0 \in (0, 1)$ where $r_X(c_0) > 0$ then $r_X(1) \geq r_X(c_0) > 0.$

In [13] it was proved that $r_X(c)$ is continuous so if $r_X(1) > 0$ then by the continuity of $r_X(c)$ there exists a $c$ such that $0 < c < 1$ and $r_X(c) > 0.$

\end{proof}

Out  of simple curiosity a more accurate bound on such a $c$ can be found.  Again in [13] the following inequality was proved for all $0 < c_1 \leq c_2$
\[ r_X(c_2) - r_X(c_1) \leq (c_2 - c_1) \frac{r_X(c_1)  + 1}{c_1}. \qquad \dag \]

\begin{lemma}
Let $X$ be a Banach space with $r_X(1) > 0.$  If $1 > c > \frac{\displaystyle 1}{\displaystyle r_X(1) + 1}$ then $r_X(c) > 0.$
\end{lemma}

\begin{proof}
Let $1 > c > \frac{\displaystyle 1}{\displaystyle r_X(1) + 1}$ then 
\begin{align*}
c(r_X(1) + 1)   & > 1 \\
r_X(1)  &  > \frac{1 - c}{c} > 0 \\
\frac{c}{1 - c}r_X(1) & > 1.  \qquad \dag\dag 
\end{align*}

In $\dag$ let $c_2 = 1$ and $c_1 =  c$ then
\begin{align*}
r_X(1) - r_X(c) & \leq (1 - c) \frac{r_X(c)  + 1}{c} \\
r_X(1) & \leq \frac{1 - c}{c}(r_X(c) + 1) + r_X(c)  \\
0 < \frac{c}{1 - c}r_X(1) & \leq r_X(c) + 1 +  \frac{c}{1 - c}r_X(c) \\
& = \frac{r_X(c)}{1 - c} + 1 \\
\frac{c}{1 - c}r_X(1) - 1 & \leq  \frac{r_X(c)}{1 - c}.
\end{align*}

This last inequality combined with the one in $\dag\dag$ gives $r_X(c) > 0.$

\end{proof}

\begin{proposition}
Let $X$ be a separable Banach space with $X^*$ having property($K^*$) then $X$ has the w-FPP.
\end{proposition}

\begin{proof}

Let $X$ be a separable Banach space with $X^*$ having property($K^*$). So $r_{X^*}(1) > 0$.

The usual set up is to assume $X$ does not have the w-FPP and arrive at a contradiction.  Using results from Goebel [10], Karlovitz [11] and Lin [12] plus an excursion to $l_\infty(X)/c_0(X)$ and then back to $X$ it can be shown that, given $t \in (0, 1),$ there exists a sequence, $w_n,$ with the following properties.

\begin{enumerate}
\item $ w_n \rightharpoonup w$
\item $\| w \| \leq 1 - t$
\item $\limsup_{n \rightarrow \infty}\| w_n - w \| = t$
\item $D[(w_n)] = \limsup_{m \rightarrow \infty}\limsup_{n \rightarrow \infty}\| w_m - w_n \| = t$
\item $\lim_{n \rightarrow \infty}\| w_n \| \geq t.$
\end{enumerate}

We need to show that $\lim_{n \rightarrow \infty}\| w_n \|$ is uniformly away from 1.

This approach can be found in Garc\'{i}a-Falset [9], Dom\'{i}nguez-Benavides [1] and Dowling, Randrianantoanina and Turett [2].  The latter paper goes a bit further and has some additional inequalities that look interesting.

For each $n \geq 1$ choose $w_n^* \in X^* \mbox{ such that } \| w_n^* \| = 1 \mbox{ and }w_n^*(w_n) = \| w_n \|$.  Using the property that $X$ is separable we can assume, without loss of generality, that $w_n^* \stackrel{*}{\rightharpoonup} w^* \mbox { where } \| w^* \| \leq 1$.

{\bf Remark:} We need either $w^*$  `deep' within the dual unit ball or $w_n^* - w^*$ eventually `deep' within the dual unit ball. The condition $r_{X^*}(1) > 0$ ensures this.
\begin{align*}
 t & = \limsup_{m \rightarrow \infty}\limsup_{n \rightarrow \infty}\| w_m - w_n \| \\
 & \geq | \limsup_{m \rightarrow \infty}\limsup_{n \rightarrow \infty} w_n^*(w_m - w_n) | \\
 & = | \limsup_{m \rightarrow \infty} w^*(w_m) - \lim_{n \rightarrow \infty}\| w_n \| | \\
 & = | w^*(w) - \lim_{n \rightarrow \infty}\| w_n \| |. \qquad \#
\end{align*}

If $\lim_{n \rightarrow \infty} \| w_n \| \leq w^*(w)$ then $\lim_{n \rightarrow \infty}\| w_n \| \leq \| w^* \| \| w \| \leq 1 - t$ and we are done.  We have the required contradiction.

If $w^* = 0$ then $\lim_{n \rightarrow \infty} \| w_n \| \leq t$ and we again done. So assume
\[ \lim_{n \rightarrow \infty} \| w_n \| > w^*(w) \mbox{ and } w^* \neq 0. \]
Another possibility that can be discarded is $w_n^* \rightarrow w^*$.  Because if $w_n^*$ is norm convergent then
\begin{align*}
0 & = \lim_{n \rightarrow \infty}\| w_n^* - w^* \|\| w_n - w \| \\
& \geq \lim_{n \rightarrow \infty} (w_n^* - w^*)(w_n - w) \\
& = \lim_{n \rightarrow \infty}\| w_n \| - w^*(w).
\end{align*}

Thus $\lim_{n \rightarrow \infty}\| w_n \| \leq w^*(w) \leq 1 - t.$
 
Now back to the main thread of the proof.  From \# we have the inequality
\[ \lim_{n \rightarrow \infty}\| w_n \| \leq t + w^*(w). \qquad \#\# \]

The proof is now split into two cases.
 
{\bf Case 1}    $\quad \liminf_{n \rightarrow \infty}\| w_n^* - w^* \| \geq \| w^* \|.$
 
Here
\[ \frac{w_n^* - w^*}{\| w^* \|}\stackrel{*}{\rightharpoonup} 0 \mbox{ and }\liminf_{n \rightarrow \infty} \left \| \frac{w_n^* - w^*}{\| w^* \|}\right \| \geq 1. \]
Thus
\[ \frac{1}{\| w^* \|} = \frac{\lim_{n \rightarrow \infty}\| w_n^* \|}{\| w^* \|} = \lim_{n \rightarrow \infty}\left \| \frac{w_n^* - w^*}{\| w^* \|} + \frac{w^*}{\| w^* \|}\right \| \geq r_{X^*}(1) + 1. \]

This leads to 
\[ \| w^* \| \leq \frac{1}{r_{X^*}(1) + 1}. \]

Using this inequality in \#\# we have
\begin{align*}
\lim_{n \rightarrow \infty}\| w_n \| & \leq t + w^*(w) \\
& \leq t + \| w^* \| \| w \| \\
& \leq t +\frac{1}{r_{X^*}(1) + 1}(1 - t) \\
& < 1.
\end{align*}
We now have the required contradiction.

{\bf Case 2} $\quad \liminf_{n \rightarrow \infty}\| w_n^* - w^* \| < \| w^* \|.$

For ease of notation let $a: = \liminf_{n \rightarrow \infty}\| w_n^* - w^* \|.$  Thus $a > 0.$

Then
       \[ \frac{w_n^* - w^*}{a} \stackrel{*} \rightharpoonup 0, \]
       \[ \liminf_{n \rightarrow \infty}\left \| \frac{w_n^* - w^*}{a}\right \| = 1 \mbox{ and } \]
       \[ 1 < \frac{\| w^* \|}{a}. \]
So
\begin{align*}
r_{X^*}(1) + 1 & \leq
\liminf_{n \rightarrow \infty}\left \| \frac{w_n^* - w^*}{a} + \frac{w^*}{a}\right \| \\
& = \lim_{n \rightarrow \infty}\left \| \frac{w_n^*}{a}\right \|.
\end{align*}

Thus $1 = \lim_{n \rightarrow \infty}\| w_n^* \| \geq  a(r_{X^*}(1) + 1).$

That is
\[ a = \liminf_{n \rightarrow \infty}\| w_n^* - w^* \| \leq \frac{1}{r_{X^*}(1) + 1}. \]

Hence
\begin{align*}
\frac{t}{r_{X^*}(1) + 1 } & \geq \liminf_{n \rightarrow \infty}\| w_n^* - w^* \|\| w_n - w \| \\
& \geq \liminf_{n \rightarrow \infty}(w_n^* - w^*)(w_n - w) \\
& = \lim_{n \rightarrow \infty}\| w_n \| - w^*(w).
\end{align*}
So
\begin{align*}
\lim_{n \rightarrow \infty}\| w_n \| & \leq \frac{t}{r_{X^*}(1) + 1} + w^*(w) \\
& \leq \frac{t}{r_{X^*}(1) + 1} + \| w^* \|\| w \| \\
& \leq \frac{t}{r_{X^*}(1) + 1} + (1 - t) \\
&  <  1.
\end{align*}

And we have the final contradiction.

Thus $X$ has the w-FPP. 

\end{proof}

{\bf Remarks}

\begin{enumerate}

\item[1.] In [15] Sims showed that Property($K$) implies weak normal structure.  That proof readily transfers to the dual. So when $X^*$ has Property($K^*$) it has weak* normal structure and hence the w*-FPP.

\item[2.] Dalby and Sims in [8] showed that if $X$ is a Banach space and $r_X(1) > 0$ then $c_0 \mathrel{\not \hookrightarrow} X.$  So when $X^*$ has Property($K^*$) then $c_0 \mathrel{\not \hookrightarrow} X^*.$  The implications resulting from $c_0$ not being in $X^*$ are unclear.

\item[3.] A very similar proof to the one above proves $R(X) \leq 1 + \frac{\displaystyle 1}{\displaystyle r_{X^*}(1) + 1} < 2$.  See [5].  Garci\'{a}-Falset in [9] proved that $R(X) < 2$ implies the w-FPP.  So this provides another route to the same conclusion.

\item[4.] As mentioned within the proof, $r_{X^*}(1) > 0$ ensures either $w^*$ is deep within the dual unit ball or $w_n^* - w^*$ eventually deep within the dual unit ball.  It may seem natural to by-pass $r_{X^*}(1) > 0$ and use two new moduli that achieves this aim directly.  Say
\[\qquad \qquad D(X) := \sup\{ \liminf_{n \rightarrow \infty}\| x_n - x \|: x_n \rightharpoonup x, \| x_n \|  = 1 \mbox{ for all } n\}.\]

and
\[\qquad \qquad \hat{D}(X) := \sup\{ \| x \|: x_n \rightharpoonup x, \| x_n \| = 1 \mbox{ for all } n\}.\]
Then, using the above proof, if either $D(X^*) < 1$ or $\hat{D}(X^*) < 1$ then $X$ has the w-FPP.  Here $D(X^*)$ and $\hat{D}(X^*)$ are defined using weak* convergent sequences.  Unfortunately, it was shown in [6] that either one of $D(X^*) < 1$ or $\hat{D}(X^*) < 1$ implies Property($K^*$) so there is no advantage in going down that path.

\end{enumerate}

\end{document}